\documentclass[12pt]{amsart}
\usepackage{fullpage,url,amssymb,enumerate,colonequals}

\usepackage{mathrsfs} 
\usepackage{MnSymbol}
\usepackage{extarrows}
\usepackage{lscape}
\usepackage[all,cmtip]{xy}

\usepackage[OT2,T1]{fontenc}

\usepackage{color}

\usepackage[
        colorlinks, citecolor=darkgreen,
        backref,
        pdfauthor={Nuno Freitas}, 
]{hyperref}

\usepackage{comment}

\newcommand{\C}{\mathbb{C}}
\newcommand{\F}{\mathbb{F}}

\newcommand{\Q}{\mathbb{Q}}

\newcommand{\Z}{\mathbb{Z}}
\newcommand{\Qbar}{{\overline{\Q}}}

\newcommand{\calO}{\mathcal{O}}

\newcommand{\fp}{\mathfrak{p}}
\newcommand{\fq}{\mathfrak{q}}

\DeclareMathOperator{\End}{End}

\DeclareMathOperator{\Frob}{Frob}

\newcommand{\GL}{\operatorname{GL}}

\numberwithin{equation}{section}

\newtheorem{theorem}{Theorem}
\newtheorem{lemma}{Lemma}
\newtheorem{corollary}{Corollary}

\theoremstyle{definition}
\newtheorem{definition}[equation]{Definition}

\theoremstyle{remark}
\newtheorem{remark}[equation]{Remark}

\definecolor{darkgreen}{rgb}{0,0.5,0}

\begin{document}

\title{A result on the equation $x^p + y^p = z^r$ using Frey abelian varieties}

\author{Nicolas Billerey}
\address{(1) Universit\'e Clermont Auvergne, Universit\'e Blaise Pascal,
    Laboratoire de Math\'ematiques,
    BP 10448,
    F-63000 Clermont-Ferrand, France.
    (2) CNRS, UMR 6620, LM, F-63171 Aubi\`ere, France}
\email{Nicolas.Billerey@math.univ-bpclermont.fr}

\author{Imin Chen}

\address{Department of Mathematics, Simon Fraser University\\
Burnaby, BC V5A 1S6, Canada } \email{ichen@sfu.ca}

\author{Luis Dieulefait}

\address{Departament d'Algebra i Geometria,
Universitat de Barcelona,
G.V. de les Corts Catalanes 585,
08007 Barcelona, Spain}
\email{ldieulefait@ub.edu}

\author{Nuno Freitas}

\address{
University of British Columbia,
Department of Mathematics,
Vancouver, BC V6T 1Z2
Canada }
\email{nunobfreitas@gmail.com}

\date{\today}

\keywords{Fermat equations, Frey abelian varieties, Irreducibility}
\subjclass[2010]{Primary 11D41}

\thanks{We acknowledge the financial support of CNRS and ANR-14-CE-25-0015 Gardio (N.~B.), an NSERC Discovery Grant (I.~C.), and the grant {\it Proyecto RSME-FBBVA $2015$ Jos\'e Luis Rubio de Francia} (N.~F.)}

\begin{abstract}
We prove a diophantine result on generalized Fermat equations of the
form $x^p + y^p = z^r$ which for the first time requires the use of Frey
abelian varieties of dimension~$\geq 2$ in Darmon's program. For
that, we provide an irreducibility criterion for the mod $\fp$
representations attached to certain abelian varieties of
$\GL_2$-type over totally real fields.
\end{abstract}

\maketitle

\section{Introduction}
Darmon \cite{darmon}  has initiated a remarkable program to study the generalized Fermat equation
\begin{align}
x^p + y^q = z^r, \qquad 1/p + 1/q + 1/r < 1, \qquad x,y,z \in \Z,
\qquad xyz \ne 0, \qquad \gcd(x,y,z) = 1, \label{E:GFE}
\end{align}
where the exponents $p, q, r \geq 2$ are prime numbers. He divides
the analysis of this equation into the three one-parameter families
$(r,r,p)$, $(p,p,r)$ and $(r,q,p)$ where in each case the parameter
$p$ is allowed to vary and the other exponents are fixed. A notable
feature of his program is that it uses higher dimensional abelian
varieties and their (still mostly conjectural) modularity instead of
just elliptic curves. However, very little is understood about the
relevant abelian varieties and Darmon's program has not yet produced
any diophantine result, apart from a few cases where the abelian
varieties involved are of dimension one, i.e.\ elliptic curves.

Darmon's program follows the strategy of the `modular method': the
Frey abelian variety $A(x,y,z)$ attached to a non-trivial (i.e.\
$xyz \ne 0$) putative solution $(x,y,z)$ of \eqref{E:GFE} can be
distinguished from the abelian varieties attached to the known
trivial solutions (i.e.\ $xyz = 0$) through their galois
representations. Indeed, the $p$-torsion representation attached to
$A(x,y,z)$ should be large in general, while if $(x,y,z)$ is a
trivial solution then this image is usually reducible or contained
in the normalizer of a Cartan subgroup. Modularity of the abelian
varieties $A(x,y,z)$ and level lowering results imply a congruence
mod $p$ between eigenforms, which bounds $p$ under the set up
described above. Another interesting feature of Darmon's program is
the use of classical cyclotomic criteria to eliminate the
possibility of a congruence to an $\mathfrak{r}$-Eisenstein
$\Q$-form at the lower levels \cite[Proposition 3.20]{darmon}.

The objective of this work is twofold. We first
develop an irreducibility criterion for the $p$-torsion
representations attached to certain families of abelian varieties.
Secondly, by following the idea in the previous paragraph and results from
{\it loc.\ cit.}, we will show how the criterion can be used to unconditionally
establish a diophantine statement via Darmon's program
that for the first time requires Frey abelian
varieties of dimension $\geq 2$.

We recall that an odd prime number $r$ is called {\it regular} if it
does not divide the class number of the cyclotomic field
$\Q(\zeta_r)$. It is an open conjecture due to Siegel that there are
infinitely many regular primes.  We will prove the following
theorem.

\begin{theorem}
Let $r \ge 5$ be a regular prime. There exists a constant $C(r)$
such that for every prime number~$p > C(r)$ the equation
\begin{equation}
x^p + y^p = z^r \label{eppr}
\end{equation}
has no non-trivial (i.e.\ $abc \not= 0$) proper (i.e.\ $\gcd(a,b,c)
= 1$) solutions $(a,b,c) \in \Z^3$ satisfying $r \mid ab$ and $2
\nmid ab$. \label{GFLT}
\end{theorem}

\noindent{\bf Acknowledgements.} Dieulefait would like to thank
Eknath Ghate for several useful conversations during an early phase
of this project.

\section{An irreducibility criterion}

The following terminology has been introduced by Ribet.
\begin{definition} An abelian variety $A$ over a number field $K$ is said to be of $\GL_2$-type if
its endomorphism algebra $\End_K(A) \otimes \Q = F$ is a number field satisfying $[F:\Q] = \dim A$.
\end{definition}
Let $A/K$ be an abelian variety of $\GL_2$-type. Set $F=\End_K(A) \otimes \Q$ and let $p$ be a prime number. Denote by $T_p(A)$ the Tate module of~$A$ and write $V_p(A)=T_p(A)\otimes\Q_p$. Then, for each prime ideal~$\fp$ of~$F$ over~$p$, the absolute Galois group~$G_K$ of~$K$ acts $F_\fp$-linearly on $V_\fp(A)=V_p(A)\otimes_{F_p}F_\fp$ where $F_\fp$ denotes the completion of~$F$ at~$\fp$ and $F_p=F\otimes\Q_p=\prod_{\fp\mid p} F_\fp$. This gives rise to a strictly compatible system of $2$-dimensional $p$-adic Galois representations
\[
\widetilde{\rho}_{A,\fp} \; \colon \; G_K\longrightarrow\GL_2(F_\fp).
\]
The representation~$\widetilde{\rho}_{A,\fp}$ can be conjugated to take values in~$\GL_2(\calO_\fp)$ where~$\calO_\fp$ stands for the ring of integers in~$F_\fp$. By reduction modulo the maximal ideal, we then get a representation
\[
\rho_{A,\fp} \; \colon G_K \; \longrightarrow\GL_2(\F_\fp),
\]
with values in the residue field~$\F_\fp$ of~$F_\fp$ which is unique up to semi-simplification and isomorphism.

The aim of this section is to provide a uniform bound on the residual characteristic of prime ideals~$\fp$ for which the corresponding representations~$\rho_{A,\fp}$ is reducible when $A$ runs through certain families of abelian varieties of $\GL_2$-type. For elliptic curves over totally real fields, such irreducibility criteria were previously known and different variants (for various families of curves) can be found in the work of Serre \cite{serre}, Kraus \cite{KrausQuad, KrausIrred}, Billerey \cite{Bil11}, David \cite{DavidI}, Dieulefait-Freitas \cite{DF2} and Freitas-Siksek \cite{FS1}.

Recently, Larson and Vaintrob \cite{larson-vaintrob} have proven general results which classify the so-called associated mod $p$ characters of abelian varieties $A$ over a number field $K$ for $p$ sufficiently large. Their results have consequences to proving irreducibility criteria for the representations~$\rho_{A,\fp}$ which we discuss here with a view towards applications to Frey abelian varieties.

For that purpose, we introduce some useful definitions.
\begin{definition}
Let $A/K$ be an abelian variety with potentially good reduction at a
prime~$\mathfrak{q}$ of a number field $K$. We say that {\it $A$ has
residual degree $f$ at $\mathfrak{q}$} if $f$ is minimal among the
degree of the residual extensions corresponding to all extensions
$L/K_\mathfrak{q}$ such that $A/L$ has good reduction.
\end{definition}

The following definition is motivated by \cite[Lemma~4.6]{larson-vaintrob}.

\begin{definition}
We say that an abelian variety $A/K$ {\it has inertial exponent $c
\in \mathbb{N}$} if for every finite prime~$v$ of the number field
$K$, there exists a finite galois extension $M/K$ such that $A/M$ is
semistable at all primes of $M$ lying over~$v$, and the exponent of
the inertia subgroup at~$v$ of $\mathrm{Gal}(M/K)$ divides~$c$.
\end{definition}

We write $\overline{\Z}$ for the ring of integers of $\Qbar$. Given
an ideal $\fq$ of the ring of integers of a number field $K$, we
write $N(\fq)$ for its norm.

\begin{theorem}
\label{main-irred-A}
Let $K$ be a totally real number field and fix a prime~$\mathfrak{q}$ of $K$. Let $c, f \ge 1$ be integers with $c$ even. Consider a finite set $S_f(\mathfrak{q})$ of elements of the form $\alpha_1 + \alpha_2$ where $\alpha_i \in \overline{\Z}$ are of complex absolute value $N(\mathfrak{q})^{f/2}$ and $\alpha_1 \alpha_2 = N(\mathfrak{q})^f$.

Then there exists a constant $c_1 = c_1(K, c, f, S_f(\mathfrak{q}))$ such that the following
holds. Suppose that $p > c_1$ and $A/K$ is an abelian variety satisfying
\begin{enumerate}[(i)]
\item\label{item1} $A$ is semistable at the primes of $K$ above $p$,
\item\label{item2} $A$ is of $\GL_2$-type with multiplications by some totally real field $F$,
\item\label{item3} all endomorphisms of~$A$ are defined over~$K$, that is $\End_K(A)=\End_{\overline{K}}(A)$,
\item\label{item4} $A$ over $K$ has inertial exponent $c$,
\item\label{item5} $A$ has potentially good reduction at $\mathfrak{q}$ with residual degree $f$,
\item\label{item6} the trace of $\Frob_\mathfrak{q}^f$ acting on $V_\fp(A)$ lies in
$S_f(\mathfrak{q})$, where $\fp$ is a prime of $F$ above $p$.
\end{enumerate}
Then the representation $\rho_{A,\fp}$ is irreducible.
\end{theorem}

\begin{remark}
\label{frobenius}
Let $L/K_\fq$ be an extension with residual degree $f$ such that $A$ over $L$ has good reduction. Let $\fq'$ be the maximal ideal of~$L$. Then $\Frob_{\fq'} = \Frob_{\fq}^f$ and hence the characteristic polynomial of $\rho_{A,\fp}(\Frob_{\fq}^f)$ is well-defined.
\end{remark}

\begin{remark}
\label{iteration} In the application to the generalized Fermat
equation, we will take $S_f(\fq)$ to be the set of traces of
$\Frob_{\fq}^f$ on $V_\fp(A(x,y,z))$, where $A(x,y,z)$ is a Frey
abelian variety attached to a solution $(x,y,z)$ of $x^p + y^p =
z^r$. To compute this set, we let $(x,y,z)$ run through all triples
in $(\Z/q\Z)^3$ such that $x^p+y^p \equiv z^r \pmod q$ and
$A(x,y,z)$ has potentially good reduction at $\fq$ lying over $q$,
for some choice of $\fq$, $c$, and $f$.
\end{remark}

\begin{proof}[Proof of Theorem~\ref{main-irred-A}]
Let $A$ be an abelian variety satisfying conditions
(\ref{item1})-(\ref{item5}) in the statement. Suppose that
$\rho_{A,\mathfrak{p}}$ is reducible. Let $\psi_i \; \colon \; G_K
\rightarrow \F_\mathfrak{p}^\times$, for $i = 1, 2$, denote the two
diagonal characters of~$\rho_{A,\mathfrak{p}}$. Then each $\psi_i$
is an associated mod~$p$ character of~$A$ of degree~$1$ in the sense
of \cite[page 518]{larson-vaintrob}. Since $A$ has inertial
exponent~$c$, then $\psi_i^c$ is unramified at all primes $v \nmid
p$ of $K$ by \cite[Proposition 3.5 (iv)]{sga7}. Moreover, since by
assumption~(\ref{item4}) $c$ is even, $\psi_i^c$ is unramified at
infinity.

We note that in {\it loc.\ cit.} a quantity $c = c(g)$ is used,
however, the proofs of the results there are still valid as long as
the $A$ in question has inertial exponent $c$ which is even.

We identify $\psi_i$ with a character of the ideles using the
reciprocity map of global class field theory. Let $\theta^S$ be
defined as in \cite[Definition~2.6]{larson-vaintrob} (with $L=\Qbar$
in their notation), where $S \in \Z[\Gamma_K]$ and $\Gamma_K$ is the
set of embeddings of $K$ into $\Qbar$.

By \cite[Lemma 5.4]{larson-vaintrob} and the semistability
assumption~(\ref{item1}), there exists $S_i\in \Z[\Gamma_K]$ such
that $\psi_i(x_{\hat p})^c \equiv \theta^{S_i}(x)^{c} \pmod
{\mathfrak{p}}$ for all $x \in K^\times$ relatively prime to $p$,
where $x_{\hat p}$ is the prime to $p$-part of $x$ regarded as an
id\`ele of $K$.

We note that the invocation of \cite[Lemma 5.4]{larson-vaintrob}
requires $p \nmid \Delta_K$, where $\Delta_K$ is the absolute
discriminant of $K$, because the proof of this lemma uses
\cite[Lemma 4.10]{larson-vaintrob}. However, the condition $p \nmid
c$ is not necessary as we assume semi-stability at $p$
by~(\ref{item1}), and hence there is no need to use \cite[Lemma
4.8]{larson-vaintrob}.

Let $B_\text{char}(K,c)$ be as given in \cite[\S 7.2, p.\
548]{larson-vaintrob}. For $p \nmid B_\text{char}(K,c)$,
$\theta^{S_i}$ is balanced by \cite[Lemma 2.15, Lemma~5.6 and \S
7.2]{larson-vaintrob}. As $K$ is totally real, a balanced character
for $K$ means being a power of the norm character \cite[Definition
2.13]{larson-vaintrob}. Thus, $\theta^{S_i}$ is a non-negative power
of the norm character.

From (\ref{item2}) $F$ is totally real, and from~(\ref{item3}) $A$
has all of its endomorphisms defined over~$K$. Hence \cite[Lemma
4.5.1]{ribet} says that we have
\begin{equation}
\label{ribet-det}
  \det \rho_{A,\fp} = \psi_1 \psi_2 = \text{cyc}_p,
\end{equation}
where $\text{cyc}_p$ denotes the mod~$p$ cyclotomic character. Thus, $\theta^{S_i}$ is either trivial or the norm character and $\theta^{S_1} \theta^{S_2}$ is the norm character. Hence, by switching $\psi_1$ for $\psi_2$ if necessary, we may assume $\psi_1^c$, is unramified at all primes of $K$.

Let $\iota : \F_{\mathfrak{p}}^\times \rightarrow \C^\times$ be an
injective homomorphism. Then $\iota \circ \psi_1$ is unramified at a
prime $v$ of $K$ if and only if $\psi_1$ is unramified at $v$. The
group of continuous characters of $G_K$ with values in $\C^\times$
which are unramified at all primes of $K$ are dual to the class
group of $K$. Hence, we have that $(\iota \circ \psi_1^c)^{h_K'} =
1$ where $h_K'$ is the exponent of the class group of $K$. Thus,
$\psi_1^{c h_K'} = 1$. By (\ref{item5}), (\ref{item6}), and
Remark~\ref{frobenius}, we obtain that
\begin{equation*}
\label{bound-BC}
p \mid \prod_{a \in S_f(\mathfrak{q})} \text{Res}(X^{c h_K'} - 1, X^2 - a X + N(\mathfrak{q})^f),
\end{equation*}
where $\text{Res}$ denotes the resultant. Therefore, letting $c_1$
denote a constant larger than any prime dividing
$B_{\text{char}}(K,c)$, $\Delta_K$, and the above resultant, gives
the desired bound.
\end{proof}

\begin{corollary} Let $K$ be a totally real field, $\fq$ a prime of $K$ and $g$ a positive integer.
There is a constant $C(K,g,\fq)$ such that the following holds:
Suppose $p > C(K,g,\fq)$ is a prime. Then for all $g$-dimensional
abelian varieties $A/K$ satisfying conditions
$(\ref{item1})-(\ref{item3})$ in Theorem~$\ref{main-irred-A}$ the
representation $\rho_{A,\fp}$ is irreducible. \label{C:main}
\end{corollary}

\begin{proof}
Since $A$ achieves semi-stable reduction over $K(A[12])$ by
\cite[Proposition 4.7]{sga7}, and the degree of the galois extension
$K(A[12])/K$ is bounded in terms of $g$, this bounds the possible
residual degrees of $A$ at $\fq$ and inertial exponents of $A$ in
terms of $g$.

Let $c_{K,g}$ be the product of all the possible inertial exponents
from the above paragraph.

If $A$ has residual degree $f$ at the prime $\fq$ of $K$, then the
characteristic polynomial of $\Frob_\fq^f$ on $T_\fp(A)$ divides the
characteristic polynomial of $\Frob_\fq^f$ on $T_p(A)$. If the
dimension of $A$ is fixed, then by \cite[Lemma 7.6]{larson-vaintrob}
there are only finitely many possibilities for the latter. Hence,
for each possible $f$ from the first paragraph, take $S_f(\fq)$ to
be the set of traces of the finitely many possibilities for the
characteristic polynomial of $\Frob_\fq^f$ on $T_\fp(A)$.

For each $f$ apply Theorem~\ref{main-irred-A} with $S_f(\fq)$ and $c = c_{K,g}$
to get a bound $c_f = c(K,c_{K,g}, f, S_f(\mathfrak{q}))$. The corollary
follows by letting $C(K,g,\fq)$ be the maximum of the $c_f$.

\end{proof}

\begin{remark}
There is an alternate method to deduce irreducibility which follows more directly from \cite[Corollary 5.19]{larson-vaintrob}. We instead picked the proof above for two reasons.
On the one hand, it is more natural as an extension of the proofs known for the case of elliptic curves and, on the other hand, since it uses properties that are normally satisfied by Frey abelian varieties, it should be better suited to giving simpler bounds in concrete diophantine applications.
\end{remark}

\section{Application to $x^p + y^p = z^r$}

In this section we use the irreducibility criterion from the previous section
to establish an unconditional diophantine statement as an application of
Darmon's program \cite{darmon} which requires Frey abelian varieties
of dimension~$\geq 2$.

\medskip

For an odd prime $r$, let $\zeta_r$ be a primitive $r$-th root of
unity and denote by~$K$ the maximal totally real subfield of
$\Q(\zeta_r)$. Let~$(a,b,c) \in \Z^3$ be a non-trivial proper
solution of~\eqref{eppr}. Put $t=a^p/c^r$ and consider the abelian
variety $J_r^+(t)$ defined in Section~1.3 of~\cite{darmon}.
According to Eq.~(5) in \emph{loc. cit.} one has
\[
\mathrm{End}_{\overline{K}}\left(J^+_r(t)\right)=\mathcal{O}_K.
\]

In particular, $J_r^+(t)$ becomes of $\GL_2$-type over $K$ with real multiplication
by $K$ (see also \cite{ttv}). Let $J_r^+(a,b,c)$ be the $\Q$-model of $J_r^+(a^p/c^r)$ defined in \cite[p.18]{darmon}.

The following two results follow from (the proof of) Proposition~1.15,  Theorem~3.22 and Definition~3.6
of~\emph{loc. cit.}.

\begin{lemma} Let $(a,b,c) \in \Z^3$ be a non-trivial proper solution to $x^p + y^p = z^r$. Suppose $r \mid ab$.
Then the abelian variety $J_r^+(a,b,c) / K$ is semistable. Moreover, if $2 \nmid ab$ it has
good reduction at all primes above~$2$.
\label{multred}
\end{lemma}

\begin{theorem}
\label{reducible-plus} Let $r$ be a regular prime. Then there exists
a constant $c_2(r)$ such that, for all $p > c_2(r)$, and non-trivial
proper solutions $(a,b,c) \in \Z^3$ to \eqref{eppr} with $r \mid
ab$, the mod $\mathfrak{p}$ representation $\rho_{r,\fp}^+$
associated to $J_r^+(a,b,c)$ is reducible.
\end{theorem}

As a consequence of these results and our irreducibility criterion in Theorem~\ref{main-irred-A}
we can now prove our main diophantine application.

\begin{proof}[Proof of Theorem~\ref{GFLT}]
Let $(a,b,c) \in \Z^3$ be a non-trivial proper solution to $x^p +
y^p = z^r$ satisfying $r \mid ab$ and $2 \nmid ab$. Write $J =
J_r^+(a,b,c) / K$. From Lemma~\ref{multred}, we have that $J$ is
semistable with good reduction at all $\fq \mid 2$. In particular,
for $J$ we have even inertial exponent $c=2$ and residual degree
$f=1$ at all $\fq \mid 2$. Let $\fq \mid 2$ in $K$. Recalling
Remark~\ref{iteration}, we take $S_f(\fq)$ to be the singleton set
consisting of the trace of $\Frob_\fq$ on the $\fp$-torsion of $J =
J^+_r(1,-1,0)$.

From Theorem~\ref{main-irred-A} we obtain a constant $c_1(r)$ such
that if $p > c_1(r)$ and $\fp \mid p$ in $K$ then the mod~$\fp$
representation $\rho_{J,\fp}$ is irreducible.

From Theorem~\ref{reducible-plus} we obtain a constant $c_2(r)$ such that
if $p > c_2(r)$ and $\fp \mid p$ in $K$ then $\rho_{J,\fp}$ is reducible.

Letting $C(r)$ be the maximum of $c_1(r)$ and $c_2(r)$, we obtain a
contradiction for all exponents $p > C(r)$.
\end{proof}

\bibliographystyle{plain}
\bibliography{bibliography}

\begin{thebibliography}{10}

\bibitem{Bil11}
Nicolas Billerey.
\newblock Crit\`eres d'irr\'eductibilit\'e pour les repr\'esentations des
  courbes elliptiques.
\newblock {\em Int. J. Number Theory}, 7(4):1001--1032, 2011.

\bibitem{darmon}
Henri Darmon.
\newblock Rigid local systems, {H}ilbert modular forms, and {F}ermat's last
  theorem.
\newblock {\em Duke Math. J.}, 102(3):413--449, 2000.

\bibitem{DavidI}
Agn\`es David.
\newblock Caract\`ere d'isog\'enie et crit\'{e}res d'irr\'{e}ductibilit\'e.
\newblock {\em arXiv:1103.3892}, 2012.

\bibitem{DF2}
Luis Dieulefait and Nuno Freitas.
\newblock Fermat-type equations of signature {$(13,13,p)$} via {H}ilbert
  cuspforms.
\newblock {\em Math. Ann.}, 357(3):987--1004, 2013.

\bibitem{FS1}
Nuno Freitas and Samir Siksek.
\newblock Criteria for irreducibility of {${\rm mod}\, p$} representations of
  {F}rey curves.
\newblock {\em J. Th\'eor. Nombres Bordeaux}, 27(1):67--76, 2015.

\bibitem{sga7}
Alexander Grothendieck.
\newblock Mod\`eles de {N}\'eron et monodromie.
\newblock {\em In {\it {S}\'eminaire de {G}\'eom\'etrie {A}lg\'ebrique}}, 7,
  Expos\'e 9, 1967--1969.

\bibitem{KrausQuad}
Alain Kraus.
\newblock Courbes elliptiques semi-stables et corps quadratiques.
\newblock {\em J. Number Theory}, 60(2):245--253, 1996.

\bibitem{KrausIrred}
Alain Kraus.
\newblock Courbes elliptiques semi-stables sur les corps de nombres.
\newblock {\em Int. J. Number Theory}, 3(4):611--633, 2007.

\bibitem{larson-vaintrob}
Eric Larson and Dmitry Vaintrob.
\newblock Determinants of subquotients of {G}alois representations associated
  with abelian varieties.
\newblock {\em J. Inst. Math. Jussieu}, 13(3):517--559, 2014.
\newblock With an appendix by Brian Conrad.

\bibitem{ribet}
Kenneth~A. Ribet.
\newblock Galois action on division points of {A}belian varieties with real
  multiplications.
\newblock {\em Amer. J. Math.}, 98(3):751--804, 1976.

\bibitem{serre}
Jean-Pierre Serre.
\newblock Propri\'et\'es galoisiennes des points d'ordre fini des courbes
  elliptiques.
\newblock {\em Invent. Math.}, 15(4):259--331, 1972.

\bibitem{ttv}
Walter Tautz, Jaap Top, and Alain Verberkmoes.
\newblock Explicit hyperelliptic curves with real multiplication and
  permutation polynomials.
\newblock {\em Canad. J. Math.}, 43(5):1055--1064, 1991.

\end{thebibliography}

\end{document}